\documentclass[12pt,a4paper]{article}
\usepackage{graphicx}
\usepackage{amsfonts}
\usepackage{amsmath}
\usepackage{amssymb}
\usepackage{textcomp}
\usepackage{verbatim}

\allowdisplaybreaks

\newtheorem{theorem}{Theorem}

\newtheorem{coro}{Corollary}
\newtheorem{defi}{Definition}
\newtheorem{remark}{Remark}

\newtheorem{prop}{Proposition}
\newtheorem{example}{Example}

\newcommand{\NN}{{\mathbb N}}

\newcommand{\PP}{{\mathbb P}}

\newcommand{\ZZ}{{\mathbb Z}}

\newcommand{\FF}{{\mathbb F}}

\newcommand{\bsb}{\boldsymbol{b}}
\newcommand{\bse}{\boldsymbol{e}}

\newcommand{\bsx}{\boldsymbol{x}}
\newcommand{\bsc}{\boldsymbol{c}}

\newcommand{\pdiv}{{\rm div}}
\newcommand{\vol}{{\rm Vol}}
\newcommand{\lcm}{{\rm lcm}}
\newcommand{\cP}{{\cal P}}
\newcommand{\cL}{{\cal L}}

\newenvironment{proof}{\begin{trivlist}
    \item[\hskip\labelsep{\it Proof.}]}{$\hfill\Box$\end{trivlist}}

\newcommand{\ustirli}[2]
{
\genfrac{[}{]}{0pt}{}{#1}{#2}}

\newcommand{\stirlii}[2]
{\genfrac{\{}{\}}{0pt}{}{#1}{#2}}

\begin{document}
\title{A construction of $(t,s)$-sequences with finite-row generating matrices using global function fields}
\date{\today}
\author{Roswitha Hofer and Harald Niederreiter}
\maketitle

\begin{abstract}
For any prime power $q$ and any dimension $s \ge 1$, we present a construction of $(t,s)$-sequences in base $q$ with finite-row
generating matrices such that, for fixed $q$, the quality parameter $t$ is asymptotically optimal as a function of $s$ as 
$s \to \infty$. This is the first construction of $(t,s)$-sequences that yields finite-row generating matrices and asymptotically
optimal quality parameters at the same time. The construction is based on global function fields. We put the construction into the
framework of $(u,\bse,s)$-sequences that was recently introduced by Tezuka. In this way we obtain in many cases better
discrepancy bounds for the constructed sequences than by previous methods for bounding the discrepancy.
\end{abstract}

\section{Introduction and basic definitions} \label{sec1}

Constructing sequences with good equidistribution properties is an important problem in number theory and has applications to quasi-Monte Carlo 
methods in numerical analysis (see, e.g., \cite{DP10, niesiam}). In this context, the star discrepancy appears as an important 
measure of uniform distribution. For a given dimension $s\geq 1$, let $J$ be a subinterval of $[0,1]^s$ and let $\bsx_0,\ldots, \bsx_{N-1}$ 
be $N$ points in $[0,1]^s$ (we speak also of a \emph{point set} $\cP$ of $N$ points in $[0,1]^s$). We write 
$A(J;\cP)=A(J;\bsx_0,\ldots, \bsx_{N-1})$ 
for the number of integers $0\leq n \leq N-1$ for which $\bsx_n\in J$. Then the \emph{star discrepancy} of the point set $\cP$ consisting of the points
$\bsx_0,\ldots,\bsx_{N-1}$ is defined by 
$$D_N^*(\cP)=D_N^*(\bsx_0,\ldots, \bsx_{N-1})=\sup_{J}\left| \frac{A(J;\cP)}{N}-\vol(J)\right|,$$
where the supremum is extended over all subintervals $J$ of $[0,1]^s$ with one vertex at the origin. 
For a sequence $S$ of points $\bsx_0,\bsx_{1},\ldots$ in $[0,1]^s$, the star discrepancy of the first $N$ terms of $S$ is defined as 
$D_{N}^*(S)=D_N^*(\bsx_0,\ldots, \bsx_{N-1})$. 
We say that $S$ is a \emph{low-discrepancy sequence} if $$D_N^*(S)=O(N^{-1}(\log N)^s) \qquad \mbox{for all $N\geq 2$},$$
where the implied constant does not depend on $N$. 
This is conjectured to be the least possible order of magnitude in $N$ that can be obtained for the star discrepancy of a sequence of points
in $[0,1]^s$. 

One of the most powerful methods for constructing low-discrepancy sequences is built on the theory of $(t,s)$-sequences and $(t,m,s)$-nets 
which was developed by Niederreiter~\cite{N87} on the basis of earlier work by Sobol'~\cite{sobol} and Faure~\cite{faure}. The reader is 
referred to, e.g., the monographs~\cite{DP10} and~\cite{niesiam} for the general background on this theory. In the following, we give 
a short description of the basic notions.  
For integers $b\geq 2$ and $0\leq t\leq m$, a \emph{$(t,m,s)$-net in base $b$} is a point set of $b^m$ points in $[0,1)^s$ such that every 
elementary interval $J \subseteq [0,1)^s$ in base $b$  with volume $b^{t-m}$ contains exactly $b^t$ points of the point set. By an 
\emph{elementary interval in base $b$} we mean an interval of the form $$\prod_{i=1}^s[a_ib^{-d_i},(a_i+1)b^{-d_i})$$ with integers $d_i\geq 0$ 
and $0\leq a_i<b^{d_i}$ for $1\leq i\leq s$. 
A sequence $\bsx_0,\bsx_1,\ldots$ of points in $[0,1]^s$ is called a \emph{$(t,s)$-sequence in base $b$}, where $b\geq 2$ and $t\geq 0$ are 
integers, if for all integers $k\geq 0$ and $m>t$ the points $[\bsx_n]_{b,m}$ with $kb^m\leq n<(k+1)b^m$ form a $(t,m,s)$-net in base $b$. Here 
$[\bsx_n]_{b,m}$ denotes the coordinatewise $m$-digit truncation in base $b$ of the point $\bsx_n$.  (For a detailed description we refer the 
reader to, e.g., \cite{T93}, \cite[Section 6.1.1]{T95}, and \cite[Section~2]{XN}; the important role of the truncation was emphasized again in the
recent paper~\cite{FL}.) In this concept the integer $t$ serves as a quality parameter, and generally speaking the smaller $t$ the more uniform the 
distribution of the sequence.

Most of the known $(t,s)$-sequences are obtained by the so-called \emph{digital method} which was developed by Niederreiter~\cite{N87} 
and constructs a sequence as follows. 
Choose a dimension $s\in\NN$, a finite field $\FF_q$ with cardinality $q$, and put $Z_q=\{0,1,\ldots,q-1\} \subset \ZZ$. 
Note that $q$ must be a prime power here. Choose 
\begin{enumerate}
\item[(i)] bijections $\psi_r: Z_q \to \FF_q$ for all integers $r\geq 0$, satisfying $\psi_r(0)=0$ for all sufficiently large $r$;
\item[(ii)] elements $c^{(i)}_{j,r} \in \FF_q$ for $1\leq i\leq s$, $j\geq 1$, and $r\geq 0$; 
\item[(iii)] bijections $\lambda_{i,j}: \FF_q\to Z_q$ for $1\leq i\leq s$  and $j\geq 1$.
\end{enumerate}
For the construction of the sequence we make use of the notion of the \emph{generating matrices} $C^{(i)}:=(c^{(i)}_{j,r})_{j\geq 1,r\geq 0}
\in\FF_q^{\NN\times \NN_0}$ for $1 \le i \le s$. If these matrices $C^{(1)},\ldots,C^{(s)}$ satisfy the property that each row of each matrix
contains only finitely many nonzero entries,
then we speak of \emph{finite-row generating matrices}. The $i$th coordinate $x_n^{(i)}$ of the $n$th point $\bsx_n=(x^{(1)}_n,\ldots,x^{(s)}_n)$ 
of the sequence is computed as follows. Given an integer $n \ge 0$, let $n=\sum_{r=0}^{\infty} z_r(n)q^r$ be the digit expansion of $n$ in base
$q$, with all $z_r(n) \in Z_q$ and $z_r(n)=0$ for all sufficiently large $r$. Then for $1 \le i \le s$ we form the matrix-vector product over
$\FF_q$ given by 
$$C^{(i)}\cdot\left(\begin{matrix}\psi_0(z_0(n))\\\psi_1(z_1(n))\\\vdots \end{matrix}\right)=:\left(\begin{matrix}y_{n,1}^{(i)}\\y_{n,2}^{(i)}\\
\vdots \end{matrix}\right).
$$
Finally, we put
$$x^{(i)}_n=\sum_{j=1}^{\infty}\lambda_{i,j}(y^{(i)}_{n,j})q^{-j}.$$

The distribution of the sequence $\bsx_0,\bsx_1,\ldots$ is mainly determined by the rank structure of the generating matrices. It is well known 
that the digital method generates a $(t,s)$-sequence in base $q$ if for every integer $m>t$ and all nonnegative integers $d_1,\ldots,d_s$ such 
that $d_1+\cdots+d_s=m-t$, the $(m-t)\times m$ matrix over $\FF_q$ formed by the row vectors 
$$(c^{(i)}_{j,0},c^{(i)}_{j,1},\ldots,c^{(i)}_{j,m-1}) \in \FF_q^m$$
with $1 \le j \le d_i$ and $1 \le i \le s$ has rank $m-t$ (see \cite[Section~4.4]{DP10} and \cite[Section~4.3]{niesiam}). 

Standard $(t,s)$-sequences obtained by the digital method include the Sobol' sequences~\cite{sobol}, Faure sequences~\cite{faure}, Niederreiter
sequences~\cite{N88}, generalized Niederreiter sequences~\cite{T93}, and Niederreiter-Xing sequences \cite{NX96, XN}. Summaries of the
constructions of these sequences can be found in \cite[Chapter~8]{DP10} and~\cite{N12}. 

The digital method can be applied also for the construction of $(t,m,s)$-nets in base $q$. In this case, the generating matrices $C^{(1)},\ldots,
C^{(s)}$ are $m \times m$ matrices over $\FF_q$ (see \cite[Section 4.4.1]{DP10}).

Tezuka~\cite{T12} recently pointed out a deeper regularity in the generating matrices of the generalized Niederreiter sequences and introduced 
so-called $(u,m,\bse,s)$-nets and $(u,\bse,s)$-sequences. He also established a general discrepancy bound for $(u,\bse,s)$-sequences and 
in this way he improved the constants in the discrepancy bounds known for generalized Niederreiter sequences.

In this paper we fulfill several objectives. In Section~\ref{sec2} we give a revised version of Tezuka's definition of $(u,m,\bse,s)$-nets 
which allows us to determine a corresponding quality parameter $t$ for a $(u,\bse,s)$-sequence in the language of $(t,s)$-sequences. In 
Section~\ref{sec3} we introduce a new construction principle for generating matrices of digital $(u,\bse,s)$-sequences using global function 
fields. From the viewpoint of $(t,s)$-sequences, these new sequences have the same quality parameter $t$ as the Niederreiter-Xing sequences, 
but the construction is simpler. Moreover, a refined construction in Section~\ref{sec4} yields $(u,\bse,s)$-sequences and $(t,s)$-sequences 
with the same parameters as in Section~\ref{sec3}, but with finite-row generating matrices. Finally, in Section~\ref{sec5} we discuss upper 
bounds on the star discrepancy for the sequences constructed in this paper.

\section{A revised definition of $(u,\bse,s)$-sequences} \label{sec2}

We follow up on the recent work of Tezuka~\cite{T12} on nets and $(t,s)$-sequences, but for several reasons (for instance, in order to prove
Proposition~\ref{prop1} below) we slightly revise his approach.

\begin{defi} \label{defi1} {\rm
Let $b\geq 2$, $s\geq 1$, and $0\leq u \leq m$ be integers and let $\bse=(e_1,\ldots, e_s) \in\NN^s$ be an $s$-tuple of positive integers. 
A \emph{$(u,m,\bse,s)$-net in base $b$} is a point set $\cP$ of $b^m$ points in $[0,1)^s$ such that $A(J;\cP)=b^m \vol(J)$ for every elementary 
interval $J$ in base $b$ of the form 
$$J=\prod_{i=1}^s\left[a_ib^{-d_i},(a_i+1)b^{-d_i}\right)$$
with integers $d_i\geq 0$, $0 \le a_i < b^{d_i}$, and $e_i|d_i$ for $1\leq i\leq s$ and with $\vol(J)\geq b^{u-m}$.}
\end{defi}

\begin{remark} \label{rm1} {\rm
The classical concept of a $(u,m,s)$-net in base $b$ corresponds to the special case $\bse =(1,\ldots,1)$ in Definition~\ref{defi1}. A simple
and well-known propagation rule for nets states that a $(u,m,s)$-net in base $b$ is also a $(v,m,s)$-net in base $b$ for any integer $v$ with
$u \le v \le m$ (see \cite[Remark 4.9(2)]{DP10} and \cite[Remark~4.3]{niesiam}). It is trivial that this propagation rule is valid also for
$(u,m,\bse,s)$-nets in base $b$.}
\end{remark}

\begin{remark} \label{rm2} {\rm
Tezuka~\cite{T12} introduced Definition~\ref{defi1}, but he used the condition $\vol(J)= b^{u-m}$ instead of $\vol(J) \ge b^{u-m}$. With this
original definition, the propagation rule for $(u,m,\bse,s)$-nets in base $b$ mentioned in 
Remark~\ref{rm1} need not hold. For instance, choose $b=s=2$ and $\bse=(2,3) \in \NN^2$. 
Furthermore, fix an integer $m\geq 3$. Now take $u=m-3$. Then the only possible elementary intervals $J \subseteq [0,1)^2$ in base $2$ in
Definition~\ref{defi1} with $\vol(J)=2^{-3}$ are of the form 
$$J=[0,1)\times[a_2/2^3,(a_2+1)/2^3).$$
Thus, the net property according to Tezuka's original definition depends only on the second coordinates in the point set $\cP$. Next take $u=m-2$. 
Then the only possible elementary intervals $J \subseteq [0,1)^2$ in base $2$ in Definition~\ref{defi1} with $\vol(J)=2^{-2}$ are of the form 
$$J=[a_1/2^2,(a_1+1)/2^2)\times[0,1).$$
Thus, the net property according to Tezuka's original definition depends only on the first coordinates in $\cP$. It is now clear that with 
Tezuka's original definition, $\cP$ can be a $(u,m,\bse,2)$-net in base $2$ for $u=m-3$, but not for $u=m-2$. }
\end{remark}

\begin{defi} \label{defi2} {\rm
Let $b \ge 2$, $s \ge 1$, and $u \ge 0$ be integers and let $\bse \in \NN^s$.
A sequence $\bsx_0,\bsx_1,\ldots$ of points in $[0,1]^s$ is called a \emph{$(u,\bse,s)$-sequence in base $b$} if for all 
integers $k\geq 0$ and $m>u$ 
the points $[\bsx_n]_{b,m}$ with $kb^m\leq n<(k+1)b^m$ form a $(u,m,\bse,s)$-net in base $b$.} 
\end{defi}

Just like in Remark~\ref{rm1}, it is clear that the classical concept of a $(u,s)$-sequence in base $b$ corresponds to the case $\bse =
(1,\ldots,1)$ in Definition~\ref{defi2}.

\begin{remark} \label{rm3} {\rm 
A further advantage of Definitions~\ref{defi1} and~\ref{defi2} is that we can drop the additional condition on $u$ and $m$ in Tezuka's 
original definitions of $(u,m,\bse,s)$-nets and $(u,\bse,s)$-sequences in base $b$, namely that $m-u$ is of the form $j_1e_1 + \cdots + j_se_s$
with integers $j_1,\ldots,j_s \ge 0$, where as above $\bse =(e_1,\ldots,e_s)$.}
\end{remark}

\begin{prop} \label{prop1}
Let $b\geq 2$, $s\geq 1$, and $0\leq u\leq m$ be integers and let $\bse=(e_1,\ldots,e_s)\in\NN^s$. Then any $(u,m,\bse,s)$-net in base $b$ is 
a $(t,m,s)$-net in base $b$ with $$t=\min(u+\sum_{i=1}^s(e_i-1),m).$$
\end{prop}

\begin{proof}
Let $\cP$ be a $(u,m,\bse,s)$-net in base $b$. If $u+\sum_{i=1}^s(e_i-1)\geq m$, the statement is trivial. Now assume that 
$u+\sum_{i=1}^s(e_i-1)< m$. We consider an elementary interval $J \subseteq [0,1)^s$ in base $b$ of the form 
$$J=\prod_{i=1}^s[a_ib^{-d_i},(a_i+1)b^{-d_i})$$
with $\vol(J)\geq b^{t-m}$, i.e., with 
\begin{equation} \label{eq21}
\sum_{i=1}^s d_i \leq m-t=m-u-\sum_{i=1}^s(e_i-1).
\end{equation}
For each $i=1,\ldots,s$, let $w_i$ be the unique integer with $0 \le w_i \le e_i-1$ such that $d_i+w_i$ is a multiple of $e_i$. Then 
$$
[a_ib^{-d_i},(a_i+1)b^{-d_i})
= \bigcup_{h=1}^{b^{w_i}}[(a_ib^{w_i}+h-1)b^{-d_i-w_i},(a_ib^{w_i}+h)b^{-d_i-w_i})
$$
is the disjoint union of $b^{w_i}$ one-dimensional elementary intervals in base $b$ of length $b^{-d_i-w_i}$. Consequently, $J$ is the disjoint 
union of $s$-dimensional elementary intervals $K_h$, $1\leq h\leq b^{w_1+\cdots +w_s}=:H$, in base $b$ with 
$$
\vol(K_h) = \prod_{i=1}^s b^{-d_i-w_i}=b^{-\sum_{i=1}^s d_i-\sum_{i=1}^s w_i} 
\geq  b^{-m+u+\sum_{i=1}^s(e_i-1-w_i)}\geq b^{u-m} 
$$
for $1 \le h \le H$, where we used~\eqref{eq21} and $w_i \le e_i-1$ for $1 \le i \le s$. By Definition~\ref{defi1}, we have 
$$A(K_h;\cP)=b^m \vol(K_h) \qquad \mbox{for } 1 \le h \le H.$$
Therefore 
$$A(J;\cP)=\sum_{h=1}^HA(K_h;\cP) =b^m \sum_{h=1}^H \vol(K_h) =b^m\vol(J),$$
and so $\cP$ is a $(t,m,s)$-net in base $b$.
\end{proof}

\begin{coro} \label{coro1}
Let $b\geq 2$, $s\geq 1$, and $u\geq 0$ be integers and let $\bse=(e_1,\ldots,e_s)\in\NN^s$. 
Then any $(u,\bse,s)$-sequence in base $b$ is a $(t,s)$-sequence in base $b$ with $$t=u+\sum_{i=1}^s(e_i-1).$$
\end{coro}

\begin{remark} \label{rm4} {\rm
Tezuka~\cite{T12} has shown with his original definition of a $(u,\bse,s)$-sequence that a generalized Niederreiter sequence is a 
$(0,\bse,s)$-sequence in base $q$, where $\bse =(e_1,\ldots,e_s) \in \NN^s$ and $e_i$, $1 \le i \le s$, is the degree of the $i$th 
base polynomial over $\FF_q$ in the construction of generalized Niederreiter sequences. It is immediately seen from his proof that
a generalized Niederreiter sequence is also a $(0,\bse,s)$-sequence in base $q$ in the sense of our Definition~\ref{defi2}.  
Corollary~\ref{coro1} implies that it is a $(t,s)$-sequence in base $q$ with $t=\sum_{i=1}^s(e_i-1)$. This gives another proof of this 
well-known result from~\cite{T93}. Consider now the case of the Niederreiter sequences constructed in~\cite{N88}. They form a subfamily
of the generalized Niederreiter sequences, and so a Niederreiter sequence is also a $(0,\bse,s)$-sequence in base $q$ and a $(t,s)$-sequence
in base $q$ with $t=\sum_{i=1}^s (e_i-1)$. Dick and Niederreiter~\cite{DN} proved that this $t$-value is the exact value of the quality
parameter $t$ for Niederreiter sequences, that is, for any Niederreiter sequence no smaller value of $t$ is possible. The case of
Niederreiter sequences shows therefore that the formula for $t$ in Corollary~\ref{coro1} is in general best possible. } 
\end{remark}

The parameter $u$ of a $(u,m,\bse,s)$-net in base $q$ constructed by the digital method can be derived from the following result. 
We use the standard convention that an empty set of vectors from a vector space is linearly independent.

\begin{prop} \label{prop2}
The matrices $C^{(1)},\ldots, C^{(s)}\in\FF_q^{m\times m}$ generate a $(u,m,\bse,s)$-net in base $q$ if and only if, for any integers 
$d_1,\ldots,d_s \geq 0$ with $e_i|d_i$ for $1\leq i\leq s$ and $d_1 + \cdots + d_s \le m-u$, the collection of $d_1 + \cdots + d_s$ vectors
obtained by taking the first $d_i$ row vectors of $C^{(i)}$ for $1 \le i \le s$ is linearly independent over $\FF_q$. 
\end{prop}

\begin{proof}
This follows by a straightforward adaptation of standard arguments that can be found in \cite[Section 4.4.2]{DP10} and \cite[Section~4.3]{niesiam}. 
\end{proof}

Using Proposition~\ref{prop2}, it is easy to determine the parameter $u$ of a $(u,\bse,s)$-sequence in base $q$ obtained by the digital method. 

\begin{prop} \label{prop3}
The matrices $C^{(1)},\ldots, C^{(s)}\in\FF_q^{\NN\times\NN_0}$ generate a $(u,\bse,s)$-sequence in base $q$ if and only if, for every 
integer $m>u$ and all integers $d_1,\ldots,d_s \geq 0$ with $e_i|d_i$ for $1\leq i\leq s$ and $1 \leq d_1+\cdots + d_s \leq m-u$, the 
$(d_1+\cdots + d_s)\times m$ matrix over $\FF_q$ formed by the row vectors 
$$(c^{(i)}_{j,0},c^{(i)}_{j,1},\ldots,c^{(i)}_{j,m-1}) \in \FF_q^m$$
with $1 \le j \le d_i$ and $1 \le i \le s$ has rank $d_1+\cdots + d_s$. 
\end{prop}

\section{A construction of $(u,\bse,s)$-sequences} \label{sec3}

So far, the only verified family of $(u,\bse,s)$-sequences with $\bse \ne (1,\ldots,1)$ is that of generalized Niederreiter sequences,
according to a result of Tezuka~\cite{T12}. The construction of generalized Niederreiter sequences in~\cite{T93}, like the earlier construction
of Niederreiter sequences in~\cite{N88}, is based on rational function fields over finite fields. In this section we present a new construction
of $(u,\bse,s)$-sequences using arbitrary global function fields. Recall that a \emph{global function field} $F$ is a finite extension of the
rational function field $\FF_q(x)$ for some finite field $\FF_q$. If $\FF_q$ is algebraically closed in $F$, then $\FF_q$ is called the
\emph{full constant field} of $F$.

First we collect some basic facts and notation for global function fields. We refer to the monograph~\cite{Sti} for general background on
global function fields. Let $F$ be a global function field with full constant field $\FF_q$ and \emph{genus} $g$ and let $\PP_F$ denote the set 
of all places of $F$. We write $\deg(P)$ for the \emph{degree} of a place $P \in \PP_F$ and $\nu_P$ for the \emph{normalized discrete valuation} 
corresponding to $P$. A \emph{divisor} $D$ of $F$ is a formal sum 
$$D=\sum_{P\in\PP_F}n_PP$$
with $n_P\in\ZZ$ for all $P\in\PP_F$ and all but finitely many $n_P=0$. We write also $n_P=\nu_P(D)$.  
The \emph{degree} $\deg(D)$ of a divisor $D$ is defined by 
$$\deg(D)=\sum_{P\in\PP_F}n_P\deg(P)=\sum_{P\in\PP_F}\nu_P(D)\deg(P).$$
We say that $D_1\leq D_2$ if $\nu_{P}(D_1)\leq \nu_P(D_2)$ for all $P\in\PP_F$. 
The \emph{principal divisor} $\pdiv(f)$ of $f\in F^*$ is defined by 
$$\pdiv(f)=\sum_{P\in\PP_F}\nu_P(f)P.$$
The \emph{Riemann-Roch space} $$\cL(D) :=\{f\in F^*:\pdiv(f)+D\geq 0\}\cup\{0\}$$
is a finite-dimensional vector space over $\FF_q$. We write $\ell(D)$ for the dimension of this vector space. Obviously, $\cL(D_1)$ is a 
subspace of $\cL(D_2)$ whenever $D_1\leq D_2$. The celebrated Riemann-Roch Theorem \cite[Theorem 1.5.15]{Sti} ensures that 
$\ell(D)\geq \deg(D)+1-g$, and equality holds whenever $\deg(D)\geq 2g-1$ (see \cite[Theorem 1.5.17]{Sti}). For a rational place $P$ 
(i.e. $\deg(P)=1$), the Weierstrass Gap Theorem \cite[Theorem 1.6.8]{Sti} says that there are exactly $g$ \emph{gap numbers} $1= i_1< \cdots
<i_g\leq 2g-1$, that is, integers $i_j$, $1 \le j \le g$, such that $\ell(i_jP)=\ell\left((i_j-1)P\right)$. Note that for integers $n \ge 2g$
it is clear by the Riemann-Roch Theorem that $\ell(nP)=\ell((n-1)P)+1$.

Finally, choosing a place $P \in \PP_F$ with degree $e$ and a \emph{local parameter} $z$ at $P$, that is, an element $z\in F$ such that 
$\nu_P(z)=1$, then any $f\in F$ has a \emph{local expansion} at the place $P$ of the form 
$$f=\sum_{k=k_0}^\infty \beta_k z^k,$$
where $k_0 \in \ZZ$ with $\nu_P(f)\geq k_0$ and $\beta_k\in\FF_{q^e}$ for all $k \ge k_0$. A detailed description of how to obtain a local
expansion can be found in \cite[pp. 5--6]{NX01}.

For our construction we choose a dimension $s\in\NN$, a finite field $\FF_q$, a global function field $F$ with full constant field $\FF_q$ 
and genus $g$, and $s+1$ distinct places $P_\infty,P_1,\ldots,P_s$ of $F$ satisfying $\deg(P_\infty)=1$. The degrees of the places
$P_1,\ldots,P_s$ are arbitrary and we put $e_i= \deg(P_i)$ for $1\leq i\leq s$. 
First we construct a sequence $(y_r)_{r\geq 0}$ of elements of $F$ using the information on the dimensions $\ell(nP)$ for $n \in \NN_0$
collected above. This allows us to choose for each integer $r\geq 0$ an element
$$y_r\in\cL(n_rP_\infty) \setminus \cL((n_r-1)P_\infty).$$
Here $n_0<n_1 < n_2 <\cdots$ are the elements in increasing order of $\NN_0\setminus \{i_1,\ldots,i_g\}$, where $i_1,\ldots,i_g$ are the 
$g$ gap numbers of $P_\infty$. Note that $n_r=r+g$ for $r\geq g$. 
From the construction of the $y_r$ we see that $\nu_{P}(y_r)\geq 0$ for all places $P\in\PP_F\setminus\{P_\infty\}$
and all $r \ge 0$. Thus for $1 \le i \le s$, 
the local expansion of $y_r$ at $P_i$ (with local parameter $z_i$ at $P_i$) has the form 
$$y_r=\sum_{k=0}^\infty \beta_{k,r}^{(i)}z_i^k \qquad \mbox{with all $\beta_{k,r}^{(i)}\in\FF_{q^{e_i}}$}.$$
Using an ordered basis of $\FF_{q^{e_i}}/\FF_q$, we can identify $\beta_{k,r}^{(i)}$ with a column vector $\bsb_{k,r}^{(i)}\in\FF_q^{e_i}$. 

Now we construct generating matrices $C^{(i)}$, $1\leq i\leq s$, over $\FF_q$ columnwise by defining column $0$, column $1$, etc. 
For column $r$ ($r\geq 0$) of $C^{(i)}$, we concatenate the column vectors $\bsb_{k,r}^{(i)}$, $k=0,1,\ldots$. 
This completes the construction. 

\begin{theorem} \label{theorem1}
Let $F$ be a global function field with full constant field $\FF_q$ and let $P_{\infty},P_1,\ldots,P_s$ be $s+1$ distinct places of $F$ with
$\deg(P_{\infty})=1$. Then the matrices $C^{(1)},\ldots,C^{(s)} \in \FF_q^{\NN \times \NN_0}$ constructed above generate a $(u,\bse,s)$-sequence 
in base $q$ with $u=g$ and $\bse=(e_1,\ldots,e_s)$, where $g$ is the genus of $F$ and $e_i=\deg(P_i)$ for $1 \le i \le s$. 
\end{theorem}

\begin{proof}
We write $C^{(i)}=\left(c_{j,r}^{(i)} \right)_{j \ge 1, r \ge 0} \in \FF_q^{\NN \times \NN_0}$ for $1 \le i \le s$. According to
Proposition~\ref{prop3}, it suffices to show that for
every integer $m>g$ and all $d_1,\ldots,d_s \in\NN_0$ such that $e_1|d_1,\ldots,e_s|d_s$ and $1 \leq d_1+\cdots+d_s\leq m-g$, the 
$(d_1+\cdots+d_s)\times m$ matrix $M$ over $\FF_q$ formed by the row vectors 
$$(c^{(i)}_{j,0},c^{(i)}_{j,1},\ldots,c^{(i)}_{j,m-1}) \in \FF_q^m$$
with $1 \le j \le d_i$ and $1 \le i \le s$ has rank $d_1+\cdots+d_s$. 

In order to prove this assertion, we show that the kernel of $M$ has dimension $\leq m-(d_1+\cdots+d_s)$. 
We take an arbitrary element $(v_0,\ldots,v_{m-1}) \in \FF_q^m$ of the kernel of $M$. Then 
\begin{equation} \label{eq1}
\sum_{r=0}^{m-1}v_rc^{(i)}_{j,r}=0 \qquad \mbox{for } 1 \le j \le d_i, \ 1 \le i \le s.
\end{equation}
In the following, we make use of the element $f:=\sum_{r=0}^{m-1}v_ry_r$ of $F$. 
From the construction of the $y_r$ and the definition of $f$ we derive that 
$$f\in {\cal L}\left(n_{m-1} P_\infty \right),$$
which is a Riemann-Roch space over $\FF_q$ with dimension $m$ since the divisor $n_{m-1} P_\infty=(m+g-1)P_\infty$ has degree 
$m+g-1> 2g-1$. Note that $y_0,\ldots,y_{m-1}$ form a basis of this Riemann-Roch space, and so each element of 
$\cL \left(n_{m-1}P_{\infty} \right)$ has a unique representation as an $\FF_q$-linear combination of $y_0,\ldots,y_{m-1}$.

From~\eqref{eq1} and the construction of the matrices $C^{(1)},\ldots,C^{(s)}$ we deduce that
$$\nu_{P_i}(f)\geq \frac{d_i}{e_i} \qquad \mbox{for } 1 \le i \le s.$$
Therefore 
$$f\in {\cal L}\left(n_{m-1} P_\infty-\sum_{i=1}^s\frac{d_i}{e_i} P_i\right),$$ which is a subspace of ${\cal L}\Big(n_{m-1} P_\infty\Big)$. 
The degree of the divisor $$D:=n_{m-1} P_\infty-\sum_{i=1}^s\frac{d_i}{e_i} P_i$$ satisfies $$\deg(D)=m+g-1-(d_1+\cdots+d_s)\geq 2g-1.$$ 
By the Riemann-Roch Theorem we know that the dimension of ${\cal L}(D)$ is $m-(d_1+\cdots+d_s)$. In view of the one-to-one correspondence
between elements of $\FF_q^m$ and elements of $\cL \left(n_{m-1} P_{\infty} \right)$ noted above, this yields the desired upper bound on the 
dimension of the kernel of $M$. 
\end{proof}

\begin{coro} \label{coro2}
With the conditions and notation in Theorem~\ref{theorem1}, the
matrices $C^{(1)},\ldots,C^{(s)}$ constructed above generate a $(t,s)$-sequence in base $q$ with quality parameter $t=g+\sum_{i=1}^s(e_i-1)$. 
\end{coro}

\begin{proof}
This follows immediately from Theorem~\ref{theorem1} by applying Corollary~\ref{coro1}. 
\end{proof}

\begin{example} \label{ex1} {\rm
If $F=\FF_q(x)$ is the rational function field over $\FF_q$ (here $g=0$) and $P_\infty$ is the infinite place of $F$, then the sequence 
$(y_r)_{r\geq 0}$ consists of elements of the polynomial ring $\FF_q[x]$ and we have $\deg(y_r)=r$ for all $r \ge 0$. Furthermore, if 
$P_1,\ldots,P_s$ are places of $F$ corresponding to distinct monic irreducible polynomials over $\FF_q$ (see \cite[Section~1.2]{Sti}), 
which serve as local parameters in the local expansions, then our construction is consistent with the one introduced in~\cite{Ho12}.}
\end{example}

\begin{remark} \label{rm5} {\rm
So far, the known constructions of generating matrices for digital $(t,s)$-sequences using arbitrary global function fields proceeded rowwise. 
The rowwise approach constructs for each component $i\in\{1,\ldots,s\}$ a sequence of elements in the global function field via Riemann-Roch spaces 
\cite{NX96, XN} or via differentials~\cite{niemay}. Then the row entries of the generating matrices are determined by using the coefficients 
of the local expansions of these elements at a fixed rational place. Our columnwise method constructs a sequence of elements in the 
global function field using Riemann-Roch spaces that are built using a fixed rational place. Then the columns of the $i$th generating matrix 
$C^{(i)}$ are determined by using the coefficients of the local expansions at the place $P_i$ corresponding to the $i$th component.  
The first columnwise construction of generating matrices was introduced by Niederreiter \cite{N88b}, which was later generalized by 
Hofer \cite{HoFFA, Ho12}. Those constructions exploited properties of certain elements of the polynomial ring $\FF_q[x]$.}
\end{remark}

\begin{remark} \label{rm6} {\rm
The basic ingredients of the construction of Niederreiter-Xing sequences in~\cite{XN} are, as in our construction, a global function field
$F$ with full constant field $\FF_q$ and $s+1$ distinct places $P_{\infty},P_1,\ldots,P_s$ of $F$ with $\deg(P_{\infty})=1$ and the degrees
$e_i=\deg(P_i)$ for $1 \le i \le s$ being arbitrary. We claim that a Niederreiter-Xing sequence with these data is a $(g,\bse,s)$-sequence in base $q$, 
where $g$ is the genus of $F$ and $\bse =(e_1,\ldots,e_s)$. This is shown by an appropriate modification of the proof of \cite[Theorem~1]{XN}.
We adhere to the notation in that proof. In view of Proposition~\ref{prop3}, it suffices to show that for every integer $m > g$ and all
$d_1,\ldots,d_s \in \NN_0$ such that $e_1|d_1,\ldots,e_s|d_s$ and $1 \le d_1 + \cdots + d_s \le m-g$, the vectors $\pi_m(\bsc_j^{(i)}) \in
\FF_q^m$, $1 \le j \le d_i$, $1 \le i \le s$, are linearly independent over $\FF_q$. In the course of the proof of \cite[Theorem~1]{XN},
an auxiliary element
$$
k \in \cL \Big(D+\sum_{i=1}^s \left(\left\lfloor \frac{d_i-1}{e_i} \right\rfloor +1 \right) P_i \Big) = \cL \Big(D+\sum_{i=1}^s \frac{d_i}{e_i} P_i
\Big)
$$
is considered. Furthermore, it is shown that $\nu_{P_{\infty}}(k) \ge m+g+1$. These two facts, together with $\nu_{P_{\infty}}(D)=0$, imply that
$$
k \in \cL \Big(D+\sum_{i=1}^s \frac{d_i}{e_i} P_i -(m+g+1) P_{\infty} \Big).
$$
Since $\deg(D)=2g$, we have
$$
\deg \Big(D+\sum_{i=1}^s \frac{d_i}{e_i} P_i -(m+g+1) P_{\infty} \Big) = g+\sum_{i=1}^s d_i -m-1 < 0,
$$
and so $k=0$ by \cite[Corollary 1.4.12(b)]{Sti}. An appeal to \cite[Lemma~2]{XN} completes the proof. By Corollary~\ref{coro1}, the
Niederreiter-Xing sequence under consideration is a $(t,s)$-sequence in base $q$ with $t=g+\sum_{i=1}^s (e_i-1)$. This recovers the main
result of~\cite{XN} in a different way. This value of $t$ is the same as in Corollary~\ref{coro2}. However, it should be pointed out    
that our new construction is considerably simpler than the one in~\cite{XN}. For instance, the auxiliary positive divisor $D$ of $F$ with
$\deg(D)=2g$ and $\nu_{P_{\infty}}(D)=0$ in~\cite{XN} is not needed in our construction. }
\end{remark}

\section{Finite-row generating matrices for $(u,\bse,s)$-sequences} \label{sec4}

Columnwise constructions as in this paper have certain advantages, for instance, one may obtain finite-row generating matrices 
more easily with this approach. For more details and the motivation for finite-row generating matrices, we refer the interested reader to 
\cite{HoFFA, Ho12} and the references therein. To drive home an important point, we cannot refrain from emphasizing one great 
asset of finite-row generating
matrices, namely that the sparsity of finite-row generating matrices leads to a considerable speedup in the computation of the matrix-vector
products that form the main operation in the actual implementation of the digital method (see Section~\ref{sec1}).

In this section we explore a more refined version of the columnwise construction in Section~\ref{sec3} using global function fields. We choose 
the $y_r$ in the sequence $(y_r)_{r\geq 0}$ by using subspaces of $\cL(n_{r} P_\infty)$ and in this way we are able to construct 
finite-row generating matrices with, in a certain sense, optimal row lengths.

As in Section~\ref{sec3}, let $s\in\NN$ be a given dimension, let $\FF_q$ be a finite field, let $F$ be a global function field with full 
constant field $\FF_q$ and genus $g$, and let $P_\infty,P_1,\ldots,P_s$ be $s+1$ distinct places of $F$ satisfying $\deg(P_\infty)=1$ and 
$\deg(P_i)=e_i$ for $1\leq i\leq s$. 
We construct $(y_r)_{r\geq 0}$ as before, but in the case where $r>g$ we choose $y_r$ such that
 $$y_r\in{\cal L}\left(n_{r} P_\infty-\sum_{i=1}^s(l_i+l_{i+1}+\cdots+l_s)\frac{v}{e_i} P_i\right)=:{\cal L }(D_r)$$ and 
 $$y_r\notin {\cal L}\left(n_{r-1} P_\infty-\sum_{i=1}^s(l_i+l_{i+1}+\cdots+l_s)\frac{v}{e_i} P_i\right)=:{\cal L}(D'_{r}). $$
Here $v:=\lcm(e_1,\ldots,e_s)$, whereas the integers $l_i \ge 0$ are determined by division with remainder in the nonnegative integers as follows:
\begin{eqnarray*}
r-g&=&  sv l_s+w_s, \ \ w_s\in\{0,\ldots,sv-1 \},\\
w_s &=& (s-1)vl_{s-1}+w_{s-1}, \ \ w_{s-1}\in\{0,\ldots,(s-1)v-1\},\\
 & \vdots &  \\
w_2&=& v  l_1+w_1, \ \ w_1\in\{0,\ldots,v-1\}.\end{eqnarray*}
Note that such a $y_r$ exists since $D'_r\leq D_r$ and $\deg(D_r)\geq r+g-(r-g)=2g$ and $\deg(D'_r)=\deg(D_r)-1$. 
The rest of the construction is carried out as in  Section~\ref{sec3} and leads to the new generating matrices $C^{(1)},\ldots,C^{(s)} \in
\FF_q^{\NN \times \NN_0}$. 

From the proof of Theorem~\ref{theorem1} and from Corollary~\ref{coro2}, we see that the generating matrices $C^{(1)},\ldots,C^{(s)}$ 
constructed in this way within the digital method yield a $(g,\bse,s)$-sequence in base $q$ with $\bse=(e_1,\ldots,e_s)$ and also a 
$(t,s)$-sequence in base $q$ with $t=g+\sum_{i=1}^s(e_i-1)$. Our more refined choice of the sequence $(y_r)_{r\geq 0}$ of elements of $F$ 
leads to bounds on the number of nonzero entries in each row of the generating matrices. To describe these bounds, we define the \emph{length}
of a row $(c_0,c_1,\ldots) \in \FF_q^{\NN_0}$ with only finitely many nonzero entries by $\max \, \{l \in \NN: c_{l-1} \ne 0 \}$, with the
proviso that the length is $0$ if $c_r=0$ for all $r \ge 0$.

\begin{theorem} \label{theorem2}
The generating matrices $C^{(1)},\ldots,C^{(s)}$ of the $(g,\bse,s)$-sequence in base $q$ constructed in this section satisfy the 
property that for all $1 \le i \le s$ and $d \in \NN$, the length of the $d$th row of $C^{(i)}$ is at most  
$$
g+s v\left\lfloor \frac{d-1}{v}\right\rfloor +iv. 
$$
Here $g$ is the genus of $F$, $\bse =(e_1,\ldots,e_s)$ with $e_i=\deg(P_i)$ for $1 \le i \le s$, and $v=\lcm(e_1,\ldots,e_s)$.
\end{theorem}

\begin{proof}
We fix $i\in\{1,\ldots,s\}$ and $d\in\NN$ and let $a \in \NN_0$ and $w \in \{1,\ldots,v\}$ be such that $d=av+w$. We write
$C^{(i)}=\left(c_{j,r}^{(i)} \right)_{j \ge 1, r \ge 0} \in \FF_q^{\NN \times \NN_0}$. In order to ensure that for every 
$r\geq g+s v\left\lfloor ({d-1})/{v}\right\rfloor +iv=g+sva+iv$, the entry $c_{d,r}^{(i)}$ in the $d$th row of $C^{(i)}$ is $0$, 
it suffices to show by the construction of $C^{(i)}$ that $\nu_{P_i}(y_r)\geq (a+1)v/e_i$. (Note that $(a+1)v\geq d$.) 
Preliminarily, we consider the case where $r-g\geq sva$ and we see from the definition of $l_s$ that $l_s\geq a$. Now in the case where 
$r-g\geq sva+iv$, we have either $l_s\geq a+1$, or otherwise $l_s =a$ and $l_j\geq 1$ for at least one $j\in\{i,i+1,\ldots,s-1\}$. 
Thus, we always have $l_i + l_{i+1} + \cdots +l_s \ge a+1$. From $y_r \in \cL(D_r)$ we deduce that
$$
\nu_{P_i}(y_r) \ge (l_i + l_{i+1} + \cdots +l_s) \frac{v}{e_i} \ge \frac{(a+1)v}{e_i}
$$
and this proves our claim. 
\end{proof}

\begin{remark} \label{rm7} {\rm
In the case where the genus $g$ is $0$ and $v=1$, the matrices $C^{(1)},\ldots,C^{(s)}$ constructed in this section have asymptotically 
shortest possible row lengths in the sense of Hofer and Larcher \cite[Section~4]{HL} (compare also with~\cite[p. 590]{HoFFA},
\cite[Remark~3]{Ho12}, and~\cite{N12}).}
\end{remark}

\begin{remark} \label{rm8} {\rm
Given a prime power $q$ and a dimension $s \in \NN$, let $F$ be a global function field with full constant field $\FF_q$ such that $F$
contains at least $s+1$ rational places. Then in the construction in this section we can, besides $P_{\infty}$, also take $P_1,\ldots,P_s$
to be rational places of $F$, that is, $e_i=\deg(P_i)=1$ for $1 \le i \le s$. In this way we obtain a $(t,s)$-sequence in base $q$ with
finite-row generating matrices and with $t=g$, the genus of $F$. Now we optimize the construction in the following sense. As in
\cite[Eq. (10)]{NX96}, we define $V_q(s)$ to be the least $g \in \NN_0$ such that there exists a global function field with full constant
field $\FF_q$, genus $g$, and at least $s+1$ rational places. Then there exists a $(V_q(s),s)$-sequence in base $q$ with finite-row generating
matrices. It was shown in \cite[Theorem~4]{NX96} that $V_q(s)=O(s)$ as $s \to \infty$ with an absolute implied constant. Hence for fixed $q$,
we obtain for any dimension $s$ a $(t,s)$-sequence in base $q$ with finite-row generating matrices such that the value of the quality
parameter $t$ grows linearly in $s$ as $s \to \infty$. This growth rate of $t$ is best possible since it is well known that for any fixed
base $b$, the least $t$-value $t_b(s)$ of any $(t,s)$-sequence in base $b$ grows at least linearly as a function of $s$ as $s \to \infty$
(see \cite[Theorem~8]{NX96b} and \cite[Theorem~1]{Sch}). For the only construction of $(t,s)$-sequences in base $q$ with finite-row
generating matrices available for any $q$ and $s$ that was known previously, namely the construction in~\cite{Ho12}, the least possible
value of $t$ grows at the rate $s \log s$ as $s \to \infty$ for fixed $q$. }
\end{remark}

\section{Discrepancy bounds} \label{sec5}
 
All known upper bounds on the star discrepancy $D_N^*(S)$ of a $(t,s)$-sequence $S$ in base $b$ are of the form
\begin{equation} \label{eq51}
D_N^*(S) \le c N^{-1} (\log N)^s +O(N^{-1} (\log N)^{s-1}) \qquad \mbox{for all } N \ge 2,
\end{equation}
where the constant $c > 0$ and the implied constant in the Landau symbol depend only on $b$, $s$, and $t$. The classical formulas for $c$ were
established by Niederreiter \cite[Section~4]{N87} and they are summarized in \cite[Theorem 4.17]{niesiam}. Later, ameliorated values of $c$
were obtained by Kritzer~\cite{Kr} and Faure and Lemieux~\cite{FL}. The currently best values of $c$ are those of Faure and Kritzer~\cite{FK},
namely $c=c_{{\rm FK}}$ given by
$$
c_{{\rm FK}}=
\begin{cases} 
\frac{b^{t}}{s!} \cdot \frac{b^2}{2(b^2-1)}\left(\frac{b-1}{2\log b}\right)^s & \text{if $b$ is even,} \\ 
\frac{b^{t}}{s!} \cdot \frac{1}{2}\left(\frac{b-1}{2\log b}\right)^s & \text{if $b$ is odd.} 
\end{cases} 
$$
It follows that for every base $b \ge 2$ we have
\begin{equation} \label{eq52}
c_{{\rm FK}} \ge \frac{b^t}{s!} \cdot \frac{1}{2} \left(\frac{b-1}{2 \log b} \right)^s.
\end{equation}

By applying the signed-splitting technique of Atanassov~\cite{Ata}, Tezuka~\cite{T12} established an upper bound on the star discrepancy
$D_N^*(S)$ of a $(u,\bse,s)$-sequence $S$ in base $b$ which is of the form~\eqref{eq51} with the constant $c=c_{{\rm Tez}}$ given by
\begin{equation} \label{eq53}
c_{{\rm Tez}} = \frac{b^u}{s!} \prod_{i=1}^s \frac{\left\lfloor b^{e_i}/2 \right\rfloor}{e_i \log b} \le
\frac{b^u}{s!} \cdot \frac{b^{e_1 + \cdots + e_s}}{(2 \log b)^s e_1 \cdots e_s}.
\end{equation}
Note that our Definition~\ref{defi2} of a $(u,\bse,s)$-sequence in base $b$ is slightly stronger than Tezuka's original definition 
in~\cite{T12}, and so Tezuka's discrepancy bound is obviously valid for our concept of a $(u,\bse,s)$-sequence in base $b$ as well.

The new sequences constructed in Sections~\ref{sec3} and~\ref{sec4} are $(t,s)$-sequences in base $q$ as well as $(u,\bse,s)$-sequences
in base $q$ with suitable $t$, $u$, and $\bse$. Therefore we can apply two versions of the discrepancy bound~\eqref{eq51}, namely
the one with $c=c_{{\rm FK}}$ and the one with $c=c_{{\rm Tez}}$. It is of interest to compare these two bounds. Recall that for
these sequences we have $u=g$, the genus of the global function field $F$, as well as $\bse =(e_1,\ldots,e_s)$ and $t=g+ \sum_{i=1}^s
(e_i-1)$, where $e_i=\deg(P_i)$ for $1 \le i \le s$. Then from~\eqref{eq52} and~\eqref{eq53} we obtain
$$
\frac{c_{{\rm FK}}}{c_{{\rm Tez}}} \ge \frac{1}{2} \left(\frac{q-1}{q} \right)^s \prod_{i=1}^s e_i.
$$
Thus, if
\begin{equation} \label{eq54}
\prod_{i=1}^s e_i > 2 \left(\frac{q}{q-1} \right)^s,
\end{equation}
then the discrepancy bound with the constant $c_{{\rm Tez}}$ is better than the one with the constant $c_{{\rm FK}}$. The
condition~\eqref{eq54} will be satisfied in many cases.

Consider a typical situation where we fix a global function field $F$ with full constant field $\FF_q$ such that $F$ has at least
one rational place $P_{\infty}$. We arrange all distinct places in $\PP_F \setminus \{P_{\infty} \}$ into a list $P_1,P_2,\ldots$
in an arbitrary manner. Then we have the following result.

\begin{prop} \label{prop4}
Let $P_1,P_2,\ldots$ be a list of all distinct places in $\PP_F \setminus \{P_{\infty}\}$ as above, where $F$ is a given global function
field with full constant field $\FF_q$. Then for any $0 < \varepsilon < 1/q$ we have
$$
\prod_{i=1}^s \deg(P_i) > (\log_q s)^{((1/q) -\varepsilon)s}
$$
for all sufficiently large $s$, where $\log_q$ denotes the logarithm to the base $q$.
\end{prop}

\begin{proof}
For any $k \in \NN$, let $B_k$ be the number of places of $F$ of degree $k$ and let $T_k=\sum_{h=1}^k B_h$ be the number of places of $F$
of degree $\le k$. Let $f_0=\deg(P_{\infty})=1,f_1,f_2,\ldots$ be the nondecreasing sequence of positive integers that is obtained by
listing each $k \in \NN$ with multiplicity $B_k$. Note that $f_0,f_1,f_2,\ldots$ is also the list of the degrees of all places of $F$
in nondecreasing order. 

Throughout the proof we assume that $s$ is sufficiently large. Let $j=j(s)$ be the largest $h \in \NN$ with $T_h \le s+1$. Then 
\begin{equation} \label{eq55}
\sum_{i=1}^s \log \deg(P_i) \ge \sum_{i=1}^s \log f_i = \sum_{i=0}^s \log f_i \ge \sum_{i=0}^{T_j-1} \log f_i = \sum_{k=1}^j B_k \log k.
\end{equation}
For any $0 < \delta < 1$ we have
\begin{equation} \label{eq56}
\sum_{k=1}^j B_k \log k \ge \sum_{k=\lceil j^{\delta} \rceil}^j B_k \log k \ge \delta (T_j-T_{\lceil j^{\delta} \rceil -1}) \log j.
\end{equation}
The Prime Number Theorem for global function fields~\cite{KS} yields
\begin{equation} \label{pnt}
T_j=\frac{q}{q-1} \cdot \frac{q^j}{j}+o \Big(\frac{q^j}{j} \Big) \qquad \mbox{as } j \to \infty.
\end{equation}
Noting that with $s$ also $j=j(s)$ is sufficiently large, we deduce from~\eqref{eq55}, \eqref{eq56}, and~\eqref{pnt} that for any
$0 < \varepsilon_1 < q/(q-1)$ we have
$$
\sum_{i=1}^s \log \deg(P_i) \ge \sum_{k=1}^j B_k \log k \ge \left(\frac{q}{q-1} -\varepsilon_1 \right) \frac{q^j}{j} \log j.
$$
The definition of $j$ implies that $T_{j+1} > s+1$. Using again~\eqref{pnt}, we obtain for any $\varepsilon_2 > 0$ that
$$
s < T_{j+1} \le \left(\frac{q}{q-1} +\varepsilon_2 \right) \frac{q^{j+1}}{j+1} < \left(\frac{q^2}{q-1} +\varepsilon_2 q \right) 
\frac{q^j}{j},
$$
and so for any $0 < \varepsilon_3 < (q-1)/q^2$ we have
$$
\frac{q^j}{j} > \left(\frac{q-1}{q^2} -\varepsilon_3 \right) s.
$$
Similarly,
$$
s < T_{j+1} \le \frac{3q^{j+1}}{j+1} < q^j,
$$
and so $j > \log_q s$. By combining these inequalities, we get
$$
\sum_{i=1}^s \log \deg(P_i) > \left(\frac{q}{q-1} -\varepsilon_1 \right) \left(\frac{q-1}{q^2} -\varepsilon_3 \right) s \log \log_q s.
$$
Now for a given $0 < \varepsilon < 1/q$ we can choose $\varepsilon_1$ and $\varepsilon_3$ suitably small such that
$$
\left(\frac{q}{q-1} -\varepsilon_1 \right) \left(\frac{q-1}{q^2} -\varepsilon_3 \right) \ge \frac{1}{q} -\varepsilon.
$$
Consequently,
$$
\sum_{i=1}^s \log \deg(P_i) > \left(\frac{1}{q} -\varepsilon \right) s \log \log_q s
$$
for all sufficiently large $s$, and this implies the desired result.
\end{proof}

If we now take for the given global function field $F$ and a given $s \in \NN$ the first $s$ places $P_1,\ldots,P_s$ from the list
$P_1,P_2,\ldots$ in Proposition~\ref{prop4}, then it follows from this result that the condition~\eqref{eq54} is satisfied for all
sufficiently large $s$.

\begin{remark} \label{rm9} {\rm
In view of Remark~\ref{rm6}, the discussion in this section applies in exactly the same way to upper bounds on the star discrepancy
of the Niederreiter-Xing sequences constructed in~\cite{XN}. }
\end{remark}

\vspace{1.5cm}

\noindent
Roswitha Hofer, Institute of Financial Mathematics, University of Linz, Altenbergerstr. 69, A-4040 Linz, Austria; email: {\tt roswitha.hofer@jku.at}

\bigskip

\noindent
Harald Niederreiter, Johann Radon Institute for Computational and Applied Mathematics, Austrian Academy of Sciences, Altenbergerstr. 69,
A-4040 Linz, Austria, and Department of Mathematics, University of Salzburg, Hellbrunnerstr. 34, A-5020 Salzburg, Austria;
email: {\tt ghnied@gmail.com} 

\end{document}